\documentclass{amsart}
\usepackage[T1]{fontenc}
\usepackage{lmodern}
\usepackage{amsmath,amssymb,amsfonts,amsthm}
\usepackage{mathrsfs}
\usepackage{graphicx}
\usepackage{multicol,multirow}
\usepackage{rotating}
\usepackage{appendix}
\usepackage{xcolor}
\usepackage{mathtools,xurl}
\usepackage{microtype}
\usepackage[colorlinks=true,linkcolor=blue,citecolor=blue,urlcolor=blue]{hyperref}

\numberwithin{equation}{section}

\theoremstyle{plain}
\newtheorem{theorem}{Theorem}[section]
\newtheorem{lemma}[theorem]{Lemma}

\newtheorem{proposition}[theorem]{Proposition}

\theoremstyle{definition}

\theoremstyle{remark}
\newtheorem{remark}[theorem]{Remark}

\hypersetup{pdfauthor={Feng Guo},
  pdftitle={Approximation numbers of critical Besov embeddings},
  pdfkeywords={Besov embeddings, approximation numbers, critical smoothness}}

\DeclareMathOperator{\rank}{rank}
\DeclareMathOperator{\tr}{tr}

\DeclareMathOperator{\Id}{Id}
\allowdisplaybreaks[1]
\title[Critical Besov embeddings]{Approximation numbers of critical Besov embeddings}
\author{Feng Guo}
\address{School of Mathematics, South China University of Technology,
Guangzhou 510640, P. R. China}
\email{70207994@nuaa.edu.cn}
\date{}

\begin{document}

\begin{abstract}
We determine the approximation numbers of critical embeddings between isotropic
Besov spaces on nonempty bounded Lipschitz domains in any dimension.
For $0<p_1<2<p_2\leq\infty$ and
$s_1-s_2=d\max\{1/p_1,1-1/p_2\}>d(1/p_1-1/p_2)$,
the approximation numbers have order $N^{-1/2}(\log(e+N))^\Gamma$ for every
$0<q_1,q_2\leq\infty$, with an explicit four-case formula for $\Gamma$.
This resolves the strictly compact critical Besov approximation-number problem
identified in Vyb\'iral's overview.
For $q_1<q_2$, the logarithmic exponent depends on the ordering of
$p_1'$ and $p_2$; it can vanish for one pair $(p_1,p_2)$ and remain positive for
another with the same fine indices.
A row estimate for the approximation error gives lower bounds when the target fine index is below one.
Gaussian finite-rank approximations acting on several levels give the sharp upper bounds
when $\Gamma<1/2$. For $q_1<q_2$, matching lower bounds for the two strict orderings of the inner indices follow
from Malykhin's finite-window estimates and norm comparisons.
The results include quasi-Banach indices and the full spaces of bounded coefficient
sequences at infinite fine indices.
\end{abstract}

\maketitle

\section{Introduction}
\label{sec:introduction}

Approximation numbers measure the smallest uniform error of linear approximation with
prescribed finite rank. The fine index of a Besov space governs summation over frequency levels.
In the range considered here, sharp nonlimiting orders depend on the smoothness gap and the
integrability indices, but not on the fine indices. At the transition between the two regimes for
$p_1<2<p_2$, the neighbouring power laws agree at $N^{-1/2}$.
At this transition, those estimates leave the logarithmic factor undetermined.

In his 2008 paper, Vyb\'iral~\cite[Section~5]{ref3}
identified the missing approximation-number order in the limiting case
\[
0<p_1<2<p_2\leq\infty,\qquad
s_1-s_2=d\max\left\{\frac1{p_1},1-\frac1{p_2}\right\}.
\]
He noted that the neighbouring formulas meet at $N^{-1/2}$, referred to an earlier
logarithmic example, and expected the unknown logarithmic order to depend on
$q_1,q_2$. The question concerns strictly compact Besov embeddings on bounded Lipschitz domains.
The later overview of Gon\c{c}alves, Haroske and
Skrzypczak~\cite[Remark~2.15]{ref15}
also excludes this transition from its power-order classification.
Thus ``critical'' here refers to a transition in the rate of approximation within the
strictly compact range.

Theorem~\ref{result:1.1} determines this logarithmic correction for every pair of
fine indices in the strictly compact Besov range above.

We use the restriction spaces $B^s_{p,q}(\Omega)$ defined in Section~\ref{sec:preliminaries}. For a bounded linear
operator $T:X\to Y$ between quasi-Banach spaces, its approximation numbers are
\[
a_N(T)=\inf\{\|T-A\|: A:X\to Y\text{ bounded linear},\ \operatorname{rank}A<N\},
\qquad N\in\mathbb N.
\]
We use $1/\infty=0$, $u_+=\max\{u,0\}$, and $p'=p/(p-1)$ for $1<p<\infty$,
$1'=\infty$, $\infty'=1$. For $0<p<1$, we also set $p'=\infty$ as an index convention;
this convention does not invoke Banach duality.

\begin{theorem}\label{result:1.1}
Let $\Omega\subset\mathbb R^d$ be a nonempty bounded Lipschitz domain,
where $d\geq1$. Let $s_1,s_2\in\mathbb R$, $0<p_1<2<p_2\leq\infty$ and
$0<q_1,q_2\leq\infty$. Suppose
\begin{equation*}
s_1-s_2=d\max\left\{\frac{1}{p_1},1-\frac{1}{p_2}\right\},
\end{equation*}
\begin{equation}
s_1-s_2>d\left(\frac{1}{p_1}-\frac{1}{p_2}\right).
\label{eq:2}
\end{equation}
Set $p'_1=p_1/(p_1-1)$ if $p_1>1$, and $p'_1=\infty$ if $p_1\leq1$.
Then $2<r=\min\{p'_1,p_2\}<\infty$, and the embedding
\[
I_\Omega:B^{s_1}_{p_1,q_1}(\Omega)\longrightarrow B^{s_2}_{p_2,q_2}(\Omega)
\]
is compact. Its approximation numbers satisfy
\begin{equation*}
a_N(I_\Omega)\asymp N^{-1/2}\bigl(\log(e+N)\bigr)^\Gamma,
\qquad N\geq1,
\end{equation*}
where $1/\infty=0$ and
\begin{equation}
\Gamma=
\begin{cases}
\displaystyle\frac12+\frac{1}{q_2}-\frac{1}{q_1},
&q_1\geq q_2,\\[5pt]
\displaystyle\frac12,
&q_1<q_2,\quad p'_1=p_2=r,\\[5pt]
\displaystyle\min\left\{\frac12,
\left(1-\frac{1}{q_1}\right)_+
+\left(\frac{1}{q_2}-\frac12\right)_+\right\},
&q_1<q_2,\quad p'_1=r<p_2,\\[5pt]
\displaystyle\min\left\{\frac12,
\frac{1}{q_2}+\left(\frac12-\frac{1}{q_1}\right)_+\right\},
&q_1<q_2,\quad p_2=r<p'_1.
\end{cases}
\label{eq:4}
\end{equation}
The comparison constants may depend on the domain, the dimension and the fixed indices,
but are independent of $N$.
\end{theorem}

The normalized smoothness gap is
\begin{equation*}
\delta=s_1-s_2-\frac{d}{p_1}+\frac{d}{p_2}=\frac{d}{r}>0.
\end{equation*}
The corner $p_1\leq1$, $p_2=\infty$ has $\delta=0$ and is excluded by~\eqref{eq:2}.
The comparison constants are taken for fixed parameters and need not be uniform as
an endpoint is approached.

The dependence on the inner indices is already visible at $(q_1,q_2)=(1,2)$.
The pairs $(p_1,p_2)=(4/3,6),(1,4),(4/3,4)$ all have $r=4$, but the logarithmic
exponents are $0,1/2,1/2$, respectively. Thus the fine indices and the common value of $r$
do not determine the logarithmic correction: the position of $p_1'$ relative to $p_2$ matters.

\subsection{Earlier results and the proof}

Skrzypczak and Vyb\'iral~\cite[equation~(1)]{ref4}, following Kulanin, give a critical
lower bound of order $N^{-1/2}(\log(e+N))^{1/q}$ for inner indices $1,p$, $2<p<\infty$,
and fine indices $\infty,q$, $1\leq q\leq\infty$.
Malykhin~\cite[Theorem~1]{ref13} determines the Kolmogorov widths of $B^1_{1,\theta}[0,1]$
in $L_p[0,1]$. Related work~\cite{ref14} determines the Kolmogorov widths of $W^1_1[0,1]$ in Lebesgue spaces.
Dota's weighted embedding results~\cite[Theorems~21 and~26]{ref9}
require strict auxiliary growth conditions that fail for the critical geometric
weights here. Cobos and K\"uhn~\cite{ref10} treat generalized smoothness with a common
inner integrability index. Spatial logarithmic weights~\cite{ref12} and periodic
spaces with dominating mixed smoothness~\cite{ref11} give other related settings.

The wavelet reduction leads to the block map
\[
D_{q_1,q_2}:\left(\bigoplus_{j\geq0}\ell_{p_1}^{M_j}\right)_{\ell_{q_1}}
\longrightarrow\left(\bigoplus_{j\geq0}\ell_{p_2}^{M_j}\right)_{\ell_{q_2}},
\qquad (Dx)_j=2^{-j/r}x_j,\quad M_j\asymp2^j.
\]
Section~\ref{sec:preliminaries} gives the required block estimates, whose proofs
are collected in Appendix~\ref{app:tools}. These estimates also cover $p_1'=p_2$.

For $q_1<q_2$ and $p_1'\ne p_2$, Appendix~\ref{app:malykhin} gives matching lower bounds
by applying the finite-window estimate in Malykhin's proof, together with norm comparisons,
inclusions, convex hulls of the source balls and finite-dimensional adjoints.
When $\Gamma=1/2$, approximating each level separately gives the matching upper bound.
When $\Gamma<1/2$, Section~\ref{sec:forward} constructs Gaussian finite-rank approximants
that act on groups of levels and attain the smaller exponent. For $p_1'<p_2$, the strict
gap between the inner indices controls the covariance error at the prescribed rank;
finite-dimensional adjoints give the upper bounds for $p_2<p_1'$.

For $q_1\geq q_2$, we obtain the lower bounds from finite-window estimates with source
fine index $\infty$ and target fine index $0<q<1$.
An approximant may connect levels, so the estimates must use its total rank.
When $p_1'<p_2$, the row estimate in Section~\ref{sec:row} controls the weighted error
matrix after compression to the same set of row and column levels. Convex optimization
and a conditional Gaussian argument give this estimate, and the rank--trace inequality
then gives the lower bound. When $p_2\leq p_1'$, the column estimate in
Appendix~\ref{app:tools} serves the same purpose. Norm comparisons on these windows
give the lower bounds for all $q_1\geq q_2$, matching the upper bound obtained by
approximating each level separately.

Section~\ref{sec:complete-classification} combines the sequence estimates, including the full
spaces of bounded block sequences at infinite fine indices.
Coordinate factorizations, wavelets and extension operators then transfer
the result to bounded Lipschitz domains in Section~\ref{sec:domain-transfer}.

\section{The normalized sequence problem}
\label{sec:preliminaries}

Fix a smooth dyadic resolution $(\varphi_j)_{j\ge0}$. For a tempered distribution $F$, put
$\Delta_jF=\mathcal F^{-1}(\varphi_j\mathcal FF)$. We equip the inhomogeneous Besov spaces with the quasi-norm
\[
 \|F\|_{B^s_{p,q}(\mathbb R^d)}
 =\left\|\left(2^{js}\|\Delta_jF\|_{L^p(\mathbb R^d)}\right)_{j\ge0}\right\|_{\ell^q},
 \qquad s\in\mathbb R,\quad 0<p,q\le\infty,
\]
and define the restriction spaces by
\begin{equation*}
 B^s_{p,q}(\Omega)=\{F|_\Omega:F\in B^s_{p,q}(\mathbb R^d)\},\qquad
 \|f\|_{B^s_{p,q}(\Omega)}=\inf_{F|_\Omega=f}\|F\|_{B^s_{p,q}(\mathbb R^d)}.
\end{equation*}
These definitions include negative smoothness and infinite indices.

For blocks $E_j$, write $(\bigoplus_jE_j)_{\ell^q}$ for the space with quasi-norm
$\|(\|x_j\|_{E_j})_j\|_{\ell^q}$. For an infinite family, we assume that the blocks are complete and
satisfy a quasi-triangle inequality with a constant independent of $j$; this holds for all fixed-$p$
blocks $\ell_p^{M_j}$ below. The resulting sums are quasi-Banach spaces, and their
coordinate injections and projections are contractions. At $q=\infty$ we always use
the full space of bounded block sequences. We use the ideal property
$a_N(STR)\le\|S\|a_N(T)\|R\|$, valid also between quasi-Banach spaces.
Implicit constants are independent of $N$ and of the finite window; set
$\Lambda_N=\log(e+N)$.

Throughout the sequence arguments,
\[
 0<p_1<2<p_2\le\infty,\qquad
 2<r=\min\{p_1',p_2\}<\infty,\qquad 0<q_1,q_2\le\infty.
\]
For $M_j\asymp2^j$ consider
\begin{equation}\label{eq:7}
 D:\left(\bigoplus_{j\ge0}\ell_{p_1}^{M_j}\right)_{\ell^{q_1}}
 \longrightarrow\left(\bigoplus_{j\ge0}\ell_{p_2}^{M_j}\right)_{\ell^{q_2}},
 \qquad (Dx)_j=t_jx_j,\quad t_j=2^{-j/r}.
\end{equation}
We write $D_{q_1,q_2}^{\mathcal W}$ for its compression to a finite set of levels
$\mathcal W$, retaining the inner indices unless specified otherwise. Put
\begin{equation*}
 \chi_q=\left(\frac1{q_2}-\frac1{q_1}\right)_+,\qquad \frac1\omega=\chi_q.
\end{equation*}
Here $\omega=\infty$ if $\chi_q=0$. The scalar multiplier formula
\[
 \|\operatorname{diag}(e):\ell^{q_1}\to\ell^{q_2}\|=\|e\|_{\ell^\omega}
\]
holds for nonnegative $e$. It follows from H\"older's inequality applied to the $q_2$-powers when $q_2<q_1$
and from the supremum formula otherwise. Since $\|\ell_{p_1}^{M_j}\hookrightarrow
\ell_{p_2}^{M_j}\|=1$ and $(t_j)$ decays geometrically, $D$ is bounded and compact:
the error obtained by reproducing the first finitely many blocks tends to zero.

We collect the auxiliary estimates needed below. Their
proofs, including the finite-window versions used in the lower bounds, are in
Appendix~\ref{app:tools}.

\begin{lemma}\label{result:2.1}
For $0<p_1<2<p_2\le\infty$ and $r=\min\{p_1',p_2\}<\infty$,
\begin{equation}\label{eq:9}
 a_n(\ell_{p_1}^M\hookrightarrow\ell_{p_2}^M)
 \asymp\min\{1,M^{1/r}n^{-1/2}\},\qquad 1\le n\le M/2.
\end{equation}
The constants depend only on the inner indices.
\end{lemma}

\begin{lemma}\label{result:2.2}
For the map \eqref{eq:7}, $a_N(D)>0$ for every $N\ge1$.
\end{lemma}

\begin{proposition}\label{result:2.7}
The map \eqref{eq:7} satisfies
\begin{equation}\label{eq:26}
 N^{-1/2}\Lambda_N^{\chi_q}\lesssim a_N(D)
 \lesssim N^{-1/2}\Lambda_N^{\chi_q+1/2},\qquad N\ge2.
\end{equation}
\end{proposition}

\begin{proposition}\label{result:2.8}
If
\begin{equation*}
 p_1=r',\qquad p_2=r,\qquad 2<r<\infty,
\end{equation*}
then, for all $0<q_1,q_2\le\infty$,
\begin{equation*}
 a_N(D)\asymp N^{-1/2}\Lambda_N^{\Gamma_q},\qquad
 \Gamma_q=\chi_q+\frac12,\qquad N\ge2.
\end{equation*}
\end{proposition}

\begin{proposition}\label{result:3.3}
If $p_2=r\le p_1'$ and $0<q_2\le1$, then
\begin{equation*}
 a_N(D)\asymp N^{-1/2}\Lambda_N^{\chi_q+1/2},\qquad N\ge2.
\end{equation*}
For each $0<q\le1$, the proof of the lower bound gives a window $\mathcal W_N$ of
$L\asymp\Lambda_N$ levels such that
\begin{equation}\label{eq:column-window}
 a_N(D_{\infty,q}^{\mathcal W_N})\gtrsim N^{-1/2}L^{1/2+1/q}.
\end{equation}
The approximants in both statements may have arbitrary entries between different levels.
\end{proposition}

For finite matrices, $S_1$ and $S_2$ denote the trace and Hilbert--Schmidt norms.
The following lemma specifies the window of levels used in the row estimate
and shows how the resulting Hilbert--Schmidt bound yields a lower bound for the approximation numbers.
All ranks refer to the entire matrix, before taking a principal compression.

\begin{lemma}\label{result:trace-window}
Fix $\Gamma\ge0$ and a sufficiently large constant $C_*$. There is a fixed stride
$h$ such that, for all sufficiently large $N$, one can choose levels
\begin{equation*}
 j_s=j_1+(s-1)h,\qquad m_s=M_{j_s},\quad t_s=2^{-j_s/r},\quad 1\le s\le L,
\end{equation*}
such that
\begin{equation}\label{eq:37}
 L\asymp\Lambda_N,\quad m_1\asymp C_*N/L^2,\quad
 m_{s+1}\ge3^rm_s,\quad t_s\asymp m_s^{-1/r},\quad
 \sum_s m_s^{-1}\lesssim m_1^{-1},
\end{equation}
and
\begin{equation}\label{eq:39}
 t_L\gtrsim N^{-1/2}L^\Gamma.
\end{equation}
On this window, let $D_L=\bigoplus_st_sI_{m_s}$ have any of the inner and outer
indices used in this paper. Let $B=\bigoplus_sb_sI_{m_s}$ and
$C=\bigoplus_sc_sI_{m_s}$ be positive weights with $b_sc_st_s=m_s^{-1}$.
Suppose every error $E=D_L-A$, $\rank A<N$, admits a set of levels $S$ with
$|S|\ge7L/8$, used for both rows and columns, such that
\begin{equation}\label{eq:trace-input}
 \|B_SE_{SS}C_S\|_{S_2}\le K\|E\|L^{1-\Gamma},
\end{equation}
where $K$ is independent of $N$ and $C_*$. Then, choosing $C_*$ large enough,
$a_N(D_L)\gtrsim N^{-1/2}L^\Gamma$.
\end{lemma}

\section{The row estimate}
\label{sec:row}

When $p_1'<p_2$ and the target fine index is below one, we need a Hilbert--Schmidt estimate for a principal compression of the weighted error matrix. We first prove this estimate when the source fine index is infinite, then use norm comparisons on a finite window to obtain the lower bound for $q_1\geq q_2$.

\begin{lemma}\label{result:3.4}
Let $0<q<1$ and $L\ge 1$. Let $H_s$, $1\le s\le L$, be real or complex Hilbert spaces,
let $\Omega_{\ell}$ be nonempty finite sets, and let $v_{\ell s,\nu}\in H_s$,
$\nu\in\Omega_{\ell}$. Define
\[
 T:\left(\bigoplus_{s=1}^L H_s\right)_{\ell^{\infty}}\longrightarrow
 \left(\bigoplus_{\ell=1}^L\ell_{\infty}^{\Omega_{\ell}}\right)_{\ell^q},
 \qquad (Tx)_{\ell,\nu}=\sum_{s=1}^L\langle x_s,v_{\ell s,\nu}\rangle,\quad M=\|T\|.
\]
In the complex case, take the inner product to be linear in the first argument. For every fixed $0<\eta<1$,
there is a set $S\subset\{1,\ldots,L\}$, used for both rows and columns, with $|S|\ge(1-\eta)L$, such that
\begin{equation}
 \sum_{\ell,s\in S}\frac1{|\Omega_{\ell}|}\sum_{\nu\in\Omega_{\ell}}\|v_{\ell s,\nu}\|^2
 \le C_{q,\eta}M^2 L^{1-2/q}.
 \label{eq:64}
\end{equation}
The constant $C_{q,\eta}$ depends only on $q$ and $\eta$.
\end{lemma}

\begin{proof}
If $M=0$, every vector is zero, and we take all labels. Suppose $M>0$, and put $\epsilon=\eta/4$.

We first delete a small set of labels to bound the sizes of the remaining vectors.
We then choose weights by convex minimization and construct random test vectors to bound
the minimum in terms of the operator norm. Finally, two further deletions turn this
weighted bound into~\eqref{eq:64} on a common set of row and column labels.

\medskip\noindent\textbf{Step 1.}\quad
Put $g_{\ell s}=\max_{\nu\in\Omega_{\ell}}\|v_{\ell s,\nu}\|$. Every bipartite matching $\mathcal Q$ satisfies
\begin{equation}
 \sum_{(\ell,s)\in\mathcal Q}g_{\ell s}^q\le 2^{1-q}M^q.
 \label{eq:65}
\end{equation}
Pairs with $g_{\ell s}=0$ contribute nothing and may be omitted. For each remaining matched pair
$(\ell,s)$, choose an index $\nu_{\ell}$ attaining $g_{\ell s}$ and place a unit vector in the direction of $v_{\ell s,\nu_{\ell}}$ in source
block $s$. Multiply these source vectors by independent signs and set all other source blocks to
zero. For each selected target coordinate, condition on all signs except the sign of its matched source block.
The two possible coordinate values are $b+a$ and $b-a$, with $|a|=g_{\ell s}$. For real or complex scalars,
\[
 \frac{|b+a|^q+|b-a|^q}{2}\ge 2^{q-1}|a|^q,
\]
because $|b+a|+|b-a|\ge 2|a|$ and $u^q+v^q\ge(u+v)^q$ for $u,v\ge 0$. Sum over the distinct matched
target labels and average to obtain \eqref{eq:65}.

Set
\[
 \tau=\left(\frac{2^{1-q}}{\epsilon L}\right)^{1/q}M.
\]
The graph of edges with $g_{\ell s}>\tau$ has a maximum matching of size less than $\epsilon L$. The bipartite
matching--cover theorem gives a vertex cover of size less than $\epsilon L$. Delete each label appearing in this cover from both the source and target index sets. The remaining set $S_0$ has size $n\ge(1-\epsilon)L$, and every vector indexed by
$S_0\times S_0$ has norm at most $\tau$. The source inclusion and target projection are contractive, so the restricted operator has norm at most $M$.
Relabel the restricted array by $1,\ldots,n$.
The choice of the threshold gives
\begin{equation*}
 n\tau^q\le C_{q,\eta}M^q.
\end{equation*}

\medskip\noindent\textbf{Step 2.}\quad
We choose weights by minimizing a convex function. Put $\sigma=q/(1-q)$ and $Q_{\ell s,\nu}=\|v_{\ell s,\nu}\|^2$.
Let $W=(W_{\ell})_{\ell}$ range over the product of the probability simplexes on $\Omega_{\ell}$, and write
\[
 b_{\ell s}(W)=\sum_{\nu\in\Omega_{\ell}}W_{\ell,\nu}Q_{\ell s,\nu}.
\]
For $\alpha_{\ell}>0$, set $a=\max_{\ell}\alpha_{\ell}$ and define
\[
 \Phi(\alpha,W)=\sum_{s=1}^n\left(\sum_{\ell=1}^n\alpha_{\ell}^2 b_{\ell s}(W)+\tau^2 a^2\right)^{1/2},
 \qquad\Psi(\alpha)=\max_W\Phi(\alpha,W).
\]
The term $\tau^2a^2$ bounds the weights on sublevel sets and provides a lower bound
for the truncation scale used in Step~3.
For fixed $W$, each summand is the maximum over $j$ of Euclidean norms whose squared values are
$\sum_{\ell}b_{\ell s}(W)\alpha_{\ell}^2+\tau^2\alpha_j^2$. Hence $\Phi$ and $\Psi$ are convex in $\alpha$. For fixed $\alpha$, $\Phi$ is concave and differentiable
in $W$.

Minimize $\Psi$ subject to the convex constraint
\[
 G(\alpha):=\sum_{\ell=1}^n\alpha_{\ell}^{-\sigma}\le 1.
\]
The exponent is chosen so that $q(\sigma+1)=\sigma$. Once the constraint is active,
the optimality relation will therefore yield the $q$th-power identity~\eqref{eq:67}.
The feasible sublevel sets are compact. Indeed, feasibility implies $\alpha_{\ell}\ge 1$, and $\Psi(\alpha)\ge n\tau\max_{\ell}\alpha_{\ell}$ bounds every
coordinate from above on each sublevel set. A minimizer therefore exists. The constraint is active at a minimizer: otherwise, scaling down all the $\alpha_{\ell}$ would decrease the objective.
Taking all coordinates sufficiently large gives a strictly feasible point, so Slater's condition holds. The convex
optimality conditions give a subgradient $\beta\in\partial\Psi(\alpha)$ and a multiplier $\lambda\ge 0$ such that
\[
 \beta_{\ell}=\lambda\sigma\alpha_{\ell}^{-\sigma-1}.
\]
Homogeneity of degree one gives $\sum_{\ell}\alpha_{\ell}\beta_{\ell}=\Psi(\alpha)$. Since $G(\alpha)=1$, we obtain
\begin{equation}
 \beta_{\ell}=F\alpha_{\ell}^{-\sigma-1},\qquad F:=\Psi(\alpha),\qquad\sum_{\ell}\beta_{\ell}^q=F^q.
 \label{eq:67}
\end{equation}

We now represent this subgradient in terms of one maximizing $W$. Fix the minimizing $\alpha$ and write, for each distribution $W$,
\[
 \mu_s(W)=\sum_{\ell}\alpha_{\ell}^2 b_{\ell s}(W).
\]
All maximizing distributions give the same vector $(\mu_s(W))_s$. This vector depends affinely on $W$, and $\tau^2a^2$ is fixed. Thus, if two maximizers gave different vectors, strict concavity of the square root in at least one summand would give a larger value at their midpoint, contradicting the definition of $\Psi(\alpha)$. Denote this common vector by $(\mu_s)_s$ and put $h_s=(\mu_s+\tau^2a^2)^{1/2}$.

The subgradient rule for a compact maximum of convex functions (Danskin's rule) expresses $\beta$ as a finite convex combination
\[
 \beta=\sum_i\theta_i\gamma^{(i)},\qquad
 \theta_i\ge0,\quad\sum_i\theta_i=1,\quad
 \gamma^{(i)}\in\partial_{\alpha}\Phi(\alpha,W^{(i)}),
\]
where every $W^{(i)}$ maximizes $\Phi(\alpha,\cdot)$. Since all these distributions give the same $h_s$, their subgradients have the form
\[
 \gamma_{\ell}^{(i)}=
 \alpha_{\ell}\sum_s\frac{b_{\ell s}(W^{(i)})}{h_s}
 +\xi_{\ell}^{(i)}\tau^2a\sum_s h_s^{-1},
\]
where each $\xi^{(i)}$ is a probability vector supported on the labels with $\alpha_{\ell}=a$. Set
\[
 W^0=\sum_i\theta_iW^{(i)},\qquad \xi=\sum_i\theta_i\xi^{(i)}.
\]
The distribution $W^0$ gives the same vector $(\mu_s)_s$ and therefore also maximizes $\Phi(\alpha,\cdot)$. The first terms in the subgradient formula are linear in $W^{(i)}$, while the regularization factors are the same for all these subgradients. Consequently,
\begin{equation}
 \beta=\beta^{\mathrm{real}}+\beta^{\mathrm{reg}},\qquad
 \beta_{\ell}^{\mathrm{real}}=\alpha_{\ell}\sum_s\frac{b_{\ell s}(W^0)}{h_s},\qquad
 \beta_{\ell}^{\mathrm{reg}}=\xi_{\ell}\tau^2a\sum_s h_s^{-1},
 \label{eq:68}
\end{equation}
where $\xi_{\ell}\ge 0$, $\sum_{\ell}\xi_{\ell}=1$, and $\xi$ is supported on the labels with $\alpha_{\ell}=a$.

At $W^0$, maximality implies that, for each target label $\ell$, the number
\begin{equation}
 \alpha_{\ell}\sum_s\frac{Q_{\ell s,\nu}}{h_s}
 \label{eq:69}
\end{equation}
equals $\beta_{\ell}^{\mathrm{real}}$ at every index in the support of $W_{\ell}^0$. Indeed, the derivative of $\Phi$ with respect to the
probability at $\nu$ in that row is $\alpha_{\ell}^2\sum_s Q_{\ell s,\nu}/(2h_s)$. At a maximum on the
probability simplex, these derivatives are equal on the support and no larger at the other coordinates.

\medskip\noindent\textbf{Step 3.}\quad
We first bound $\sum_{\ell}(\beta_{\ell}^{\mathrm{real}})^q$ by $C_qM^q$,
using test vectors whose source blocks have norm at most one.
Choose the indices $\nu_{\ell}$ independently with laws $W_{\ell}^0$. Independently of these choices, take independent standard real Gaussian variables $g_{\ell}$, and set
\[
 z_s=\sum_{\ell}\alpha_{\ell}g_{\ell}v_{\ell s,\nu_{\ell}},\qquad
 x_s=\frac{z_s}{\max\{4h_s,\|z_s\|\}}.
\]
Each source block has norm at most one, so $\|Tx\|\le M$ for every outcome. In the complex case,
regard each $H_s$ as a real Hilbert space with scalar product $\operatorname{Re}\langle\cdot,\cdot\rangle$ and the same norm. Fix a target label $\ell$ and an index $\nu_{\ell}$ in the support of $W_{\ell}^0$. Condition on this index and on all index choices and Gaussian variables at the other target labels. Define $x_s(t)$ by replacing $g_{\ell}$ with $t$, and set
\[
 f_{\ell}(t)=\operatorname{Re}\sum_s\langle x_s(t),v_{\ell s,\nu_{\ell}}\rangle.
\]
The radial map $z\mapsto$
$z/\max\{4h_s,\|z\|\}$ is Lipschitz and is the gradient of a convex radial function on the real Hilbert space.
It is therefore monotone, and $f_{\ell}$ is Lipschitz and nondecreasing in $t$.

Write $w_s=\sum_{k\ne\ell}\alpha_k g_k v_{ks,\nu_k}$. Let $\mathbb E_{-\ell}$ and $\mathbb P_{-\ell}$ denote expectation and probability over the index choices and Gaussian variables at the other labels, conditional on the fixed $\nu_{\ell}$. Independence gives
\[
 \mathbb E_{-\ell}\|w_s\|^2=\sum_{k\ne\ell}\alpha_k^2 b_{ks}(W^0)\le\mu_s\le h_s^2.
\]
The event $G_s=\{\|w_s\|\le 2h_s\}$ therefore satisfies $\mathbb P_{-\ell}(G_s)\ge3/4$ and is independent of $g_{\ell}$.
By Step~1 and the regularization term in $h_s$, we have $\alpha_{\ell}\|v_{\ell s,\nu_{\ell}}\|\le a\tau\le h_s$. On $G_s$ and for all $-1\le t\le 1$,
$\|w_s+\alpha_{\ell}t v_{\ell s,\nu_{\ell}}\|\le 3h_s<4h_s$, so the contribution of this column to the derivative of $f_{\ell}$ is $\alpha_{\ell}Q_{\ell s,\nu_{\ell}}/(4h_s)$. By monotonicity, every other column contributes a nonnegative derivative at almost every $t$.

Put
\[
 B_{\ell}=\sum_s\frac{\alpha_{\ell}Q_{\ell s,\nu_{\ell}}}{h_s}\mathbf 1_{G_s}.
\]
By \eqref{eq:69}, $0\le B_{\ell}\le\beta_{\ell}^{\mathrm{real}}$ and $\mathbb E_{-\ell}B_{\ell}\ge 3\beta_{\ell}^{\mathrm{real}}/4$. When $\beta_{\ell}^{\mathrm{real}}>0$, these two bounds imply
\[
 \mathbb P_{-\ell}\{B_{\ell}\ge\beta_{\ell}^{\mathrm{real}}/2\}\ge\frac12.
\]
On this event, integration over $-1\le t\le1$ gives
$f_{\ell}(1)-f_{\ell}(-1)\ge B_{\ell}/2\ge\beta_{\ell}^{\mathrm{real}}/4$. Thus either $f_{\ell}(1)\ge\beta_{\ell}^{\mathrm{real}}/8$ or $f_{\ell}(-1)\le-\beta_{\ell}^{\mathrm{real}}/8$. By monotonicity, the same absolute lower bound holds on the corresponding tail $g_{\ell}\ge1$ or $g_{\ell}\le-1$. The tail is chosen after the other variables are fixed. Since $g_{\ell}$ is independent of them, either tail has the same fixed positive conditional probability. Integrating first over $g_{\ell}$ and then over the other labels and $\nu_{\ell}$ gives
\[
 \mathbb E|f_{\ell}(g_{\ell})|^q\ge c_q(\beta_{\ell}^{\mathrm{real}})^q.
\]
The same inequality is immediate when $\beta_{\ell}^{\mathrm{real}}=0$. Since $|f_{\ell}(g_{\ell})|\le\max_{\nu\in\Omega_{\ell}}|(Tx)_{\ell,\nu}|$ for every outcome, summing over $\ell$ and using $\|Tx\|\le M$ yields
\begin{equation}
 \sum_{\ell}(\beta_{\ell}^{\mathrm{real}})^q\le C_q M^q.
 \label{eq:70}
\end{equation}

\begin{samepage}
For the regularization term, we have
\[
 \sum_{\ell}\beta_{\ell}^{\mathrm{reg}}=\tau^2 a\sum_s h_s^{-1}\le n\tau,
 \qquad\sum_{\ell}(\beta_{\ell}^{\mathrm{reg}})^q\le n^{1-q}(n\tau)^q=n\tau^q\le C_{q,\eta}M^q.
\]
\end{samepage}
Combining this with \eqref{eq:67}, \eqref{eq:68}, and \eqref{eq:70}, we obtain $F\le C_{q,\eta}M$.

\medskip\noindent\textbf{Step 4.}\quad
Let $U_{\ell}$ be the uniform distribution on $\Omega_{\ell}$. Maximality of $W^0$ gives
\[
 \sum_s\left(\sum_{\ell}\alpha_{\ell}^2b_{\ell s}(U)\right)^{1/2}
 \le\Phi(\alpha,U)\le F\le C_{q,\eta}M.
\]
Set $\rho_{\ell}=\alpha_{\ell}^{-1}$, so $\sum_{\ell}\rho_{\ell}^{\sigma}=1$. Delete the $\lfloor\epsilon n\rfloor$ largest $\rho_{\ell}$ and call the remaining target labels $R$.
Then $\max_{\ell\in R}\rho_{\ell}\le(\epsilon n)^{-1/\sigma}$. For every source label put
\[
 u_s=\left(\sum_{\ell\in R}b_{\ell s}(U)\right)^{1/2}.
\]
The preceding estimate gives
\[
 \sum_s u_s\le C_{q,\eta}M n^{1-1/q}.
\]
Delete the $\lfloor\epsilon n\rfloor$ largest $u_s$ and intersect the remaining source labels with $R$. The resulting
set $S$ satisfies $|S|\ge(1-2\epsilon)n$ and
\[
 \sum_{\ell,s\in S}b_{\ell s}(U)\le\sum_{s\in S}u_s^2
 \le\frac1{\epsilon n}\left(\sum_s u_s\right)^2\le C_{q,\eta}M^2 n^{1-2/q}.
\]
The two rearrangement estimates use $\lfloor\epsilon n\rfloor+1>\epsilon n$ and thus also hold for small $n$.
Finally, $(1-\epsilon)(1-2\epsilon)\ge 1-\eta$ and $n\asymp_{\eta}L$, which gives \eqref{eq:64}.
\end{proof}

\begin{proposition}\label{result:3.5}
Let $L\ge 1$ be finite, let $2<r<\infty$, $p=r'$, and $0<q<1$, and suppose $0<c_0\le d_s=t_s m_s^{1/r}\le C_0$. For $1\le s,\ell\le L$, put
\[
 \begin{aligned}
 X_{\infty}&=\left(\bigoplus_s\ell_p^{m_s}\right)_{\ell^{\infty}},\quad
 &Y_q&=\left(\bigoplus_{\ell}\ell_{\infty}^{m_{\ell}}\right)_{\ell^q},\\
 B&=\bigoplus_{\ell}d_{\ell}^{-1}m_{\ell}^{-1/2}I,\quad
 &C&=\bigoplus_s m_s^{1/r-1/2}I.
 \end{aligned}
\]
For every $E:X_{\infty}\to Y_q$, there is a set $S\subset\{1,\ldots,L\}$, used for both rows and columns, with $|S|\ge 7L/8$, such that
\begin{equation*}
 \|B_S E_{SS}C_S\|_{S_2}\le C_{q,c_0}L^{1/2-1/q}\|E\|.
\end{equation*}
\end{proposition}

\begin{proof}
Each $C_s:\ell_2^{m_s}\to\ell_p^{m_s}$ is contractive, because $m_s^{1/r-1/2}m_s^{1/p-1/2}=1$. Take $H_s=\ell_2^{m_s}$,
$\Omega_{\ell}=\{1,\ldots,m_{\ell}\}$, and $v_{\ell s,\nu}=d_{\ell}^{-1}C_s E_{\ell s}^*e_{\nu}$. The operator defined in Lemma~\ref{result:3.4} then has norm at
most $c_0^{-1}\|E\|$, and the averaged sum of squared vector norms on $S$ is $\|B_S E_{SS}C_S\|_{S_2}^2$. The lemma
with $\eta=1/8$ gives the result.
\end{proof}

\begin{theorem}\label{result:3.6}
Suppose $p_1'=r<p_2\le\infty$, $2<r<\infty$, $q_1=\infty$, and $0<q=q_2<1$. The map~\eqref{eq:7} satisfies
\[
 a_N(D_{\infty,q})\asymp N^{-1/2}\bigl(\log(e+N)\bigr)^{1/q+1/2}.
\]
\end{theorem}

\begin{proof}
Proposition~\ref{result:2.7} gives the upper bound. To prove the lower bound, first
take the inner target index to be $\infty$. Choose the window in Lemma~\ref{result:trace-window}
with $\Gamma=1/q+1/2$, and use the weights $B,C$ of
Proposition~\ref{result:3.5}. The diagonal weights satisfy
$b_sc_st_s=m_s^{-1}$. Applied to an error matrix, Proposition~\ref{result:3.5} gives
\[
 \|B_SE_{SS}C_S\|_{S_2}
 \lesssim\|E\|L^{1/2-1/q}=\|E\|L^{1-\Gamma}.
\]
Lemma~\ref{result:trace-window} therefore gives the lower bound on this window.
The contractive inclusion from $\ell_{p_2}$ to $\ell_\infty$ in each target block gives the
same lower bound for finite $p_2>r$. Applying finite compression gives the bound on the full space,
and Lemma~\ref{result:2.2} covers the finitely many small indices.
\end{proof}

\subsection{\texorpdfstring{The range $q_1\ge q_2$}{The range q1 >= q2}}

\begin{proposition}\label{result:3.7}
Suppose $0<p_1<2<p_2\le\infty$, $2<r=\min\{p_1',p_2\}<\infty$, and $0<q_2\le q_1\le\infty$, with $p_1'=\infty$ when $p_1\le 1$. For the map~\eqref{eq:7},
\begin{equation*}
 a_N(D_{q_1,q_2})\asymp N^{-1/2}\bigl(\log(e+N)\bigr)^{1/2+1/q_2-1/q_1}.
\end{equation*}
Here $1/\infty=0$. The constants depend only on the fixed parameters and the
constants in $M_j\asymp2^j$.
\end{proposition}

\begin{proof}
The upper bound follows from~\eqref{eq:26}. Put $q_*=\min\{1/2,q_2\}$, with
$q_*=1/2$ if $q_2=\infty$. If $p_1'=r<p_2$, the proof of
Theorem~\ref{result:3.6} gives a finite window $\mathcal W_N$ of
$L\asymp\Lambda_N$ blocks with
\[
 a_N(D_{\infty,q_*}^{\mathcal W_N})\gtrsim
 N^{-1/2}L^{1/2+1/q_*}.
\]
If $p_2=r\le p_1'$, Proposition~\ref{result:3.3} supplies the same estimate,
including the case $p_2=p_1'$. Both estimates on the window use the inner indices
$p_1,p_2$.

On this window, let $J$ be the source inclusion from outer index $\infty$ to outer index $q_1$,
and let $K$ be the target inclusion from outer index $q_2$ to outer index $q_*$. Then
\[
 \|J\|=L^{1/q_1},\qquad \|K\|=L^{1/q_*-1/q_2},\qquad
 D_{\infty,q_*}^{\mathcal W_N}=KD_{q_1,q_2}^{\mathcal W_N}J.
\]
These norm comparisons also hold for indices below one and for infinite indices. By finite compression
and the ideal property,
\[
 a_N(D_{q_1,q_2})\ge
 L^{-1/q_1-(1/q_*-1/q_2)}a_N(D_{\infty,q_*}^{\mathcal W_N})
 \gtrsim N^{-1/2}L^{1/2+1/q_2-1/q_1}.
\]
For small $N$, use positivity of a single block compression.
\end{proof}

\section{Upper bounds by grouping levels}
\label{sec:forward}

Proposition~\ref{result:2.7} gives the separate-level
$N^{-1/2}\Lambda_N^{1/2}$ upper bound. To obtain the smaller exponents below, we approximate several levels
together using Gaussian matrices. The matching lower bounds follow
from Malykhin's finite-window argument~\cite{ref13}, as shown in
Appendix~\ref{app:malykhin}.
Throughout this section, $D_{q_1,q_2}$ is the critical map~\eqref{eq:7},
with $M_j\asymp2^j$, $t_j=2^{-j/r}$, and $q_1\le q_2$.

\begin{theorem}\label{result:4.1}
Suppose $0<p_1<2<p_2\le\infty$,
$2<r=\min\{p_1',p_2\}<\infty$, $p_1'\ne p_2$, and
$0<q_1\le q_2\le\infty$. Then
\begin{equation*}
 a_N(D_{q_1,q_2})\asymp N^{-1/2}\Lambda_N^\Gamma,\qquad N\ge2,
\end{equation*}
where
\begin{equation*}
\Gamma=\begin{cases}
\displaystyle\Gamma_s=\min\left\{\frac12,
 \left(1-\frac1{q_1}\right)_++\left(\frac1{q_2}-\frac12\right)_+\right\},
 &p_1'=r<p_2,\\[5pt]
\displaystyle\Gamma_t=\min\left\{\frac12,
 \frac1{q_2}+\left(\frac12-\frac1{q_1}\right)_+\right\},
 &p_2=r<p_1'.
\end{cases}
\end{equation*}
\end{theorem}

For $q_1<q_2$, the upper bound improves when $q_1<2$ and
$q_2>1$ in the first case, and when $q_2>2$ in the second.
The strict inequality between the inner indices allows us to approximate
several levels together using a fixed proportion of the available rank.

\subsection{Gaussian finite-rank approximation}
For maps $B:\ell_2^m\to Y$ and $C:X\to\ell_2^m$ between
finite-dimensional real Banach spaces, write
\[
 w(B)=\mathbb E\|Bg\|_Y,\qquad
 w(C)=\mathbb E\|C^*g\|_{X^*},\qquad g\sim N(0,I_m).
\]

\begin{lemma}\label{result:4.2}
Let $G$ be an $m\times n$ matrix with independent standard real
Gaussian entries. Then
\begin{equation}
\mathbb E\left\|B\left(\frac{GG^*}{n}-I_m\right)C\right\|
\le C_0\left[
 \frac{\|B\|w(C)+\|C\|w(B)}{\sqrt n}
 +\frac{w(B)w(C)}n\right].
\label{eq:81}
\end{equation}
Consequently, $BC$ has an approximant of rank at most $n$ whose
error satisfies this bound up to an absolute constant. The same statement
holds for the complex-linear extensions of real matrices between the
mixed-norm spaces used below.
\end{lemma}

\begin{proof}
Chevet's Gaussian inequality~\cite{ref1} gives, for a standard
Gaussian matrix $H$ of compatible dimensions,
\begin{equation}
 \mathbb E\|BHC\|\le\sqrt2\,[\|B\|w(C)+\|C\|w(B)].
 \label{eq:83}
\end{equation}
To see this, let $U,V$ be the images of the target dual ball
and source ball under $B^*,C$, respectively, with Euclidean radii $R_U,R_V$. Then
\[
 \|u\otimes v-u'\otimes v'\|_2^2
 \le2R_V^2\|u-u'\|_2^2+2R_U^2\|v-v'\|_2^2.
\]
Comparison with the Gaussian process
$\sqrt2R_V\langle g,u\rangle+\sqrt2R_U\langle h,v\rangle$
proves~\eqref{eq:83}; the symmetry of the two sets allows us to take
the supremum without absolute values.

Let $G'$ be an independent copy of $G$. Convexity gives
\[
 \mathbb E\|B(GG^*-nI_m)C\|
 \le\mathbb E\|B(GG^*-G'G'^*)C\|
 \le2\mathbb E\|BUV^*C\|,
\]
where $U=(G+G')/\sqrt2$ and $V=(G-G')/\sqrt2$ are independent
standard Gaussian matrices. Apply~\eqref{eq:83} conditionally on $V$
to obtain
\[
 \mathbb E_U\|BUV^*C\|
 \le\sqrt2\,[\|B\|\mathbb E_h\|C^*Vh\|+\|V^*C\|w(B)].
\]
Here $h\sim N(0,I_n)$, and, conditionally on $h$, $Vh$ has the
distribution $\|h\|_2g$. Thus
\[
 \mathbb E_{V,h}\|C^*Vh\|\le\sqrt n\,w(C),\qquad
 \mathbb E\|V^*C\|\lesssim w(C)+\sqrt n\|C\|,
\]
where the second inequality follows again from~\eqref{eq:83}.
Dividing by $n$ proves~\eqref{eq:81}. Choose a realization of
$BGG^*C/n$ with error at most twice its expected value; its rank is at
most $n$.

For complex scalars, regard the same real matrix as a complex-linear
map. The rank of a real matrix is the same over $\mathbb R$ and $\mathbb C$.
Decomposition into real and imaginary parts bounds its complex operator
norm by twice the real norm, since the target is Banach.
\end{proof}

The covariance estimate also follows from the matrix-deviation
estimates in~\cite[Theorems~1.1 and~1.4 in the author manuscript]{ref2} by polarization.

\begin{lemma}\label{result:4.4}
Let $Y$ be Banach, and let the blocks $E_j$ satisfy the standing
completeness and uniform quasi-triangle assumptions. For $0<s\le1$,
every bounded linear map $T:(\bigoplus_jE_j)_{\ell_s}\to Y$ satisfies
\[
 \|T\|=\sup_j\|T\iota_j\|.
\]
The map extends to $(\bigoplus_jE_j)_{\ell_1}$ with the same norm
and the same rank whenever its rank is finite. Hence the approximation numbers of a
diagonal map into $Y$ do not depend on the source outer exponent
$0<s\le1$. For finite collections of finite-dimensional blocks, replacing an
inner source exponent $0<p\le1$ by $1$ also preserves operator norms
into $Y$ and approximation numbers.
\end{lemma}

\begin{proof}
For $s\le1$,
\[
 \Big\|\sum_jT\iota_jx_j\Big\|_Y
 \le\sup_j\|T\iota_j\|\sum_j\|x_j\|
 \le\sup_j\|T\iota_j\|
       \Big(\sum_j\|x_j\|^s\Big)^{1/s}.
\]
The first inequality shows that the same sum defines a bounded extension
to the outer $\ell_1$-sum.
Since a finite-dimensional range is closed, the extension preserves
finite rank. Restricting and extending approximants proves the
assertion about approximation numbers.

On a finite window, write each nonzero block as $x_j=c_jz_j$,
where $c_j=\|x_j\|_1$. Express $z_j$ as a convex combination of coordinate atoms,
using unimodular factors over the complex field. Choose these atoms
independently in each nonzero block and leave the zero blocks unchanged. The resulting vectors all have the same outer
norm, with inner $\ell_p$-norm $c_j$ in each block. Convexity of the target norm
gives one inequality between the operator norms for $p$ and $1$;
the reverse inequality follows from $\|x_j\|_1\le\|x_j\|_p$. Apply this
identity to each error matrix to obtain equality of the approximation
numbers.
\end{proof}

\subsection{The upper bound when \texorpdfstring{$p_1'<p_2$}{p1prime < p2}}
\begin{proposition}\label{result:4.5}
If $p_1'=r<p_2$, $q_1\le q_2$, and $q_2\ge1$, then
\[
 a_N(D_{q_1,q_2})\lesssim N^{-1/2}\Lambda_N^{\Gamma_s}.
\]
\end{proposition}

\begin{proof}
If $q_1<1$, Lemma~\ref{result:4.4} reduces the problem to $q_1=1$,
with the same exponent $\Gamma_s$.
For $q_1\ge2$, the classical upper bound~\eqref{eq:26} gives
logarithmic exponent $1/2$. We therefore assume $1\le q_1\le2$.

First suppose $p_2<\infty$. For a consecutive group
$\mathcal J=\{a,\ldots,b\}$ of $L$ levels, apply
Lemma~\ref{result:4.2} with
\[
 \begin{aligned}
 X&=\Big(\bigoplus_{j\in\mathcal J}\ell_{r'}^{M_j}\Big)_{\ell_{q_1}},
 &Y&=\Big(\bigoplus_{j\in\mathcal J}\ell_{p_2}^{M_j}\Big)_{\ell_{q_2}},\\
 C&=D^{\mathcal J}:X\to\ell_2,\quad &B&=I:\ell_2\to Y.
 \end{aligned}
\]
The Euclidean radii and Gaussian widths satisfy
\begin{equation}
 \|C\|\le2^{-a/r},\qquad
 \|B\|=L^{(1/q_2-1/2)_+},
 \label{eq:84}
\end{equation}
\begin{equation}
 w(C)\lesssim L^{1-1/q_1},\qquad w(B)\lesssim2^{b/p_2}.
 \label{eq:85}
\end{equation}
For $q_1>1$, the first width bound follows from Gaussian moment
estimates with fixed exponents and $t_jM_j^{1/r}\asymp1$. At $q_1=1$, Gaussian
concentration and a union bound give
\[
 w(C)\le C_1+C_2\,2^{-a/r}\sqrt{\log(e+L)}.
\]
Thus~\eqref{eq:85} holds at this endpoint for the groups used below,
where $a\gtrsim\log(e+N)$ and $L\lesssim\log(e+N)$.
For the second bound, use
$w(B)\le\sum_{j=a}^b\mathbb E\|g_j\|_{p_2}\lesssim2^{b/p_2}$.
The resulting approximant has rank at most $n$ and error
\begin{equation}
 \lesssim n^{-1/2}
 \left[L^{1-1/q_1+(1/q_2-1/2)_+}+2^{b/p_2-a/r}\right]
 +n^{-1}L^{1-1/q_1}2^{b/p_2}.
 \label{eq:86}
\end{equation}

For sufficiently large $N$, reproduce all levels below $a_0$ exactly, choosing
$2^{a_0}\asymp cN$ with $c>0$ small enough that their total
dimension is at most $N/4$. Discard the levels above
$b_0=(r/2)\log_2N+O(1)$. The discarded tail has norm at most $CN^{-1/2}$
since $q_1\le q_2$. Choose $1<\theta<p_2/r$ and partition the
middle window into consecutive groups satisfying $b\le\theta a+O(1)$.
The number $K$ of groups is bounded independently of $N$, since
$b_0/a_0\to r/2$. Use $n=\lfloor N/(4K)\rfloor$ Gaussian columns
on each group. The exact map on the lower levels and the Gaussian approximants
together have rank less than $N$.

On these groups $2^{b/p_2-a/r}\lesssim1$ and
$2^{b/p_2}\lesssim N^{r/(2p_2)}$. Hence the last term
in~\eqref{eq:86} is
\[
 O\big(N^{-1+r/(2p_2)}\Lambda_N^{1-1/q_1}\big)
 =o\big(N^{-1/2}\Lambda_N^{\Gamma_s}\big),
\]
using $p_2>r$. The logarithmic exponent in the first term of~\eqref{eq:86} is
$\Gamma_s$ for $1\le q_1\le2$, $q_1\le q_2$.
Since $q_1\le q_2$, the norm of the direct sum of group errors is at
most the largest group norm. This proves the claim for finite $p_2$.

For $p_2=\infty$, retain the exact map on the lower levels and
discard the upper levels as above. Use one Gaussian approximation on
the entire middle window, with
$n=\lfloor N/4\rfloor$. The radii in~\eqref{eq:84} are unchanged,
and the Gaussian coordinate maximum gives
\[
 w(B)\lesssim L^{1/q_2}
 \sqrt{\log\Big(e+\sum_{j=a_0}^{b_0}M_j\Big)}
 \lesssim\Lambda_N^{1/q_2+1/2}.
\]
This includes $q_2=\infty$. The two terms in~\eqref{eq:81} containing
$w(B)$ have additional factors $N^{-1/r}$ and $N^{-1/2}$,
respectively, up to fixed powers of $\Lambda_N$, compared with $N^{-1/2}$.
They are therefore smaller than the leading term
$N^{-1/2}L^{\Gamma_s}$. Enlarging the constants covers small $N$.
\end{proof}

\subsection{The upper bound when \texorpdfstring{$p_2<p_1'$}{p2 < p1prime}}
\begin{proposition}\label{result:4.8}
If $p_2=r<p_1'$, $q_1\le q_2$, and $q_2\ge1$, then
\[
 a_N(D_{q_1,q_2})\lesssim N^{-1/2}\Lambda_N^{\Gamma_t}.
\]
\end{proposition}

\begin{proof}
Set $\bar p_1=\max\{p_1,1\}$ and
$\bar q_1=\max\{q_1,1\}$. The target is Banach on every finite
window. By Lemma~\ref{result:4.4}, the approximation numbers are
unchanged when the source indices are replaced by
$(\bar p_1,\bar q_1)$. On each finite window, the adjoint between the
resulting Banach spaces has inner indices $r',\bar p_1'$
with $r<\bar p_1'\le\infty$ and
outer indices $u=q_2'\le v=\bar q_1'$. Taking adjoints of
finite-rank approximants in each direction gives
$a_N(T)=a_N(T^*)$, including the $1,\infty$ endpoints.
Proposition~\ref{result:4.5} gives
\[
 \Gamma_s(u,v)=\begin{cases}
 \min\{1/2,1/q_2\},&\bar q_1\le2,\\
 1/2+1/q_2-1/\bar q_1,&2\le\bar q_1\le q_2,
 \end{cases}
\]
which equals $\Gamma_t$.

To pass to the infinite map, fix $N$ and truncate at a level $J$
for which the tail norm is at most
$N^{-1/2}\Lambda_N^{\Gamma_t}$. This tail norm is $\sup_{j>J}t_j$.
The adjoint of this finite truncation is a contractive compression
of the infinite model in Proposition~\ref{result:4.5}. Apply its upper estimate,
take the finite-dimensional adjoint of the approximant, and extend
it by zero. This preserves the rank. Apply Lemma~\ref{result:4.4} to return to
the original source indices, and add the tail bound to obtain the estimate. This also covers $p_1\le1$
and $q_2=\infty$. All adjoints are taken on finite windows,
including at an outer-$\infty$ endpoint.
\end{proof}

\begin{proof}[Proof of Theorem~\ref{result:4.1}]
For $q_2\ge1$ the upper bounds follow from
Propositions~\ref{result:4.5} and~\ref{result:4.8}.
For $q_1\le q_2\le1$, both exponents are $1/2$ and the
classical block bound~\eqref{eq:26} applies. The matching lower
bounds follow from Proposition~\ref{app:malykhin:consequences}.
For Banach fine indices, its exponents $B_s,B_t$ equal
$\Gamma_s,\Gamma_t$ when $q_1\le q_2$. Lemma~\ref{result:4.4}
covers $q_1<1\le q_2$, and the final assertion of
Proposition~\ref{app:malykhin:consequences} covers $q_2<1$.
For the finitely many values of $N$ not covered by the window constructions,
positivity of a sufficiently large finite block compression gives the lower bound.
\end{proof}

\section{Approximation numbers on sequence spaces and domains}
\label{sec:complete-classification}

\begin{theorem}\label{result:5.1}
Let $0<p_1<2<p_2\le\infty$, $2<r=\min\{p'_1,p_2\}<\infty$,
$M_j\asymp2^j$, and $t_j=2^{-j/r}$. For every $0<q_1,q_2\le\infty$,
\begin{equation*}
 a_N(D_{q_1,q_2})\asymp N^{-1/2}\Lambda_N^\Gamma,
 \qquad N\ge1,
\end{equation*}
where $\Gamma$ is given by~\eqref{eq:4}. When a fine index is infinite, we use
the full space of bounded block sequences.
\end{theorem}

\begin{proof}
For $q_1\ge q_2$, apply Proposition~\ref{result:3.7}. For $q_1<q_2$,
Proposition~\ref{result:2.8} treats $p'_1=p_2$, and
Theorem~\ref{result:4.1} treats the two strict orderings of the inner indices. For both finite
and infinite fine indices, the finite-window estimates pass to the full
spaces by the argument below. Lemma~\ref{result:2.2} handles the
finitely many remaining values of $N$.
\end{proof}

\subsection{Full bounded coefficient spaces}
\label{sec:infinite-indices}

At $q_1=\infty$, finite-window lower bounds apply to every bounded
finite-rank operator on the full source. Indeed, coordinate projection
$P_W$ and zero extension $\iota_W$ have norm one, and
\[
 \rank(P_WA\iota_W)\le\rank A,
 \qquad \|P_W(D-A)\iota_W\|\le\|D-A\|.
\]
For the upper bounds, extend each finite-window approximant using these
coordinate maps. Let $D^{\le J}$ denote the diagonal map truncated after
level $J$. Then
\begin{equation}\label{eq:91}
 \|D-D^{\le J}\|=
 \begin{cases}
  \sup_{j>J}t_j,&q_1\le q_2,\\[2pt]
  \left(\displaystyle\sum_{j>J}t_j^a\right)^{1/a},
   &q_1>q_2,\quad 1/a=1/q_2-1/q_1,
 \end{cases}
 \qquad t_j=2^{-j/r}.
\end{equation}
The identities on the inner coordinate spaces have norm one, attained at
coordinate atoms. Thus the tail norm equals the scalar diagonal norm on
the outer spaces. Both expressions are comparable to $2^{-J/r}$, also when
$q_1=\infty$. With $\vartheta=\min\{1,q_2\}$, interpreted as $1$ at
$q_2=\infty$, the power triangle inequality in the target space gives
\[
 a_N(D)^\vartheta\le a_N(D^{\le J})^\vartheta
       +\|D-D^{\le J}\|^\vartheta.
\]
Thus the uniform finite-window estimates pass to the full spaces.
In particular, for $q_1=\infty$ and $q_2<1$, the required lower estimate
comes from Theorem~\ref{result:3.6} when $p_1'<p_2$ and
Proposition~\ref{result:3.3} when $p_2\le p_1'$.

Restricting the source to the space of sequences whose block norms tend to zero,
$X_0\subset X_\infty=(\bigoplus_j\ell_{p_1}^{M_j})_{\ell_\infty}$, leaves
the approximation numbers unchanged, with target
$Y=(\bigoplus_j\ell_{p_2}^{M_j})_{\ell_{q_2}}$.
To see this, write a bounded finite-rank map on
$X_0$ as $A_0x=\sum_{i=1}^k f_i(x)y_i$, where the $f_i$ are continuous
linear functionals on $X_0$ and $y_i\in Y$. Set $f_{ij}=f_i\iota_j$.
Choosing nearly norming vectors and signs or phases on each finite set
of blocks gives
\[
 \sum_j\|f_{ij}\|\le\|f_i\|,
 \qquad f_i(x)=\sum_j f_{ij}(x_j)\quad(x\in X_0),
\]
where the identity follows by density of finite truncations. This
argument also applies when $p_1<1$. The absolutely convergent series
extends $f_i$ to $X_\infty$ and hence gives a bounded extension $A$ of
rank at most $k$. For $x\in X_\infty$, its finite truncations $x^{(J)}$ satisfy
$A_0x^{(J)}\to Ax$ in the finite-dimensional range and $Dx^{(J)}\to Dx$
by~\eqref{eq:91}. Since $\|x^{(J)}\|\le\|x\|$, passage to the limit gives
$\|D-A\|\le\|D|_{X_0}-A_0\|$. Restriction gives the reverse inequality,
and hence
\[
 a_N(D:X_\infty\to Y)=a_N(D|_{X_0}:X_0\to Y).
\]

\subsection{Geometric multiplicities}

Coordinate factorizations compare the block multiplicities in different
dimensions while preserving the approximation index.

\begin{proposition}\label{result:6.1}
Let $d\in\mathbb N$, $0<r<\infty$, and $0<p_i,q_i\le\infty$ for $i=1,2$.
Suppose $M_j\asymp2^{jd}$ and $K_k\asymp2^k$ for all sufficiently large
levels, with arbitrary finite nonnegative initial multiplicities.
Write
\[
 X_i^{(d)}=\left(\bigoplus_j\ell_{p_i}^{M_j}\right)_{\ell_{q_i}},
 \qquad
 X_i^{(1)}=\left(\bigoplus_k\ell_{p_i}^{K_k}\right)_{\ell_{q_i}},
 \qquad \ell_p^0=\{0\}.
\]
Let $D_d$ and $D_1$ multiply their respective blocks by $2^{-jd/r}$ and
$2^{-k/r}$. If one of these maps is bounded, then both are bounded, and
\begin{equation}\label{eq:92}
 a_N(D_d)\asymp a_N(D_1),\qquad N\ge1.
\end{equation}
The constants may depend on the fixed parameters and multiplicities,
but not on $N$. At infinite outer indices, we use the full spaces of bounded
block sequences.
\end{proposition}

\begin{proof}
Choose constants so that
\[
 M_j\le C_M2^{jd},\qquad K_k\le C_K2^k
\]
at every level, and $M_j\ge a_M2^{jd}$, $K_k\ge a_K2^k$ beyond fixed
initial levels. Choose a nonnegative integer $c$ beyond the threshold at which
$K_k\ge a_K2^k$ holds, with $a_K2^c\ge C_M$. Inject block $j$ into
the first $M_j$ coordinates of block $dj+c$, and denote this injection by $J$.
Let $R$ project those coordinates back and multiply them by $2^{c/r}$.
Then
\begin{equation}\label{eq:94}
 D_d=RD_1J,\qquad \|J\|\le1,\quad\|R\|\le2^{c/r}.
\end{equation}

Conversely, choose $h$ beyond the lower-bound threshold for $M_j$ with
$a_M2^{dh}\ge C_K(2^d-1)$. Since
\[
 \sum_{a=0}^{d-1}K_{di+a}\le C_K(2^d-1)2^{di}\le M_{i+h},
\]
we can place the $d$ consecutive blocks of dimensions $K_{di+a}$, $0\le a<d$,
in disjoint coordinate sets of the block of dimension $M_{i+h}$. Denote the corresponding injections and projections
by $\iota_{ia}$ and $\pi_{ia}$, and set
\[
 (Pz)_{i+h}=\sum_{a=0}^{d-1}\iota_{ia}z_{di+a},\quad
 (Pz)_j=0\ (j<h),\qquad
 (Qy)_{di+a}=2^{(dh-a)/r}\pi_{ia}y_{i+h}.
\]
The finite-dimensional comparison
\begin{equation*}
 \|b\|_{\ell_u^d}\le d^{(1/u-1/v)_+}\|b\|_{\ell_v^d},
 \qquad 0<u,v\le\infty,
\end{equation*}
applied to the disjoint coordinate sets gives
\[
 \|P\|\le d^{(1/p_1-1/q_1)_+},\qquad
 \|Q\|\le2^{dh/r}d^{(1/q_2-1/p_2)_+}.
\]
These estimates use sums of powers or maxima and therefore also hold for
indices below one and at infinity. Evaluating each coordinate gives the second
factorization
\begin{equation}\label{eq:99}
 D_1=QD_dP,
 \qquad
 (QD_dPz)_{di+a}=2^{-(di+a)/r}z_{di+a}.
\end{equation}
Both factorizations hold on the full spaces, so boundedness of either
map implies boundedness of the other. Applying the ideal
property to~\eqref{eq:94} and~\eqref{eq:99} proves~\eqref{eq:92}, with
the same $N$ in both directions. The fixed shifts account for all initial
multiplicities.
\end{proof}

\subsection{Besov embeddings on bounded Lipschitz domains}
\label{sec:domain-transfer}

Rychkov's extension operator $E_\Omega$ satisfies $R_\Omega E_\Omega=\Id$
and is bounded from $B_{p,q}^s(\Omega)$ to $B_{p,q}^s(\mathbb R^d)$ for
every $s\in\mathbb R$ and $0<p,q\le\infty$; see~\cite{ref6}
and~\cite[Lemma~6.3]{ref7}. Here $R_\Omega$ is distributional restriction.
For the two fixed Besov spaces, choose a common compactly supported
orthonormal wavelet system with $C^k$ regularity and vanishing
moments through order $k$, where
\begin{equation*}
 k>\max_{i=1,2}\left\{s_i,\frac{2d}{p_i}+\frac d2-s_i\right\}.
\end{equation*}
Its analysis and synthesis maps $W,W^{-1}$ give coefficient
isomorphisms for both Besov spaces. The wavelet expansions converge unconditionally in $\mathcal S'$, also at
infinite indices; see~\cite[Theorem~2.7]{ref3}
and~\cite[Theorem and Corollary~5]{ref5}. For $L^2$-normalized wavelets the
coefficient quasi-norm is
\begin{equation}\label{eq:101}
 \left\|\left(2^{j(s+d/2-d/p)}
       \|(c_{j,m,e})_{m,e}\|_{\ell_p}\right)_{j\ge0}\right\|_{\ell_q},
\end{equation}
with the scaling coefficients and the finitely many wavelet directions
included in the inner blocks.

\begin{proof}[Proof of Theorem~\ref{result:1.1}]
The critical identity gives
\[
 \delta=s_1-s_2-d/p_1+d/p_2
       =d\max\{1/p_2,1-1/p_1\}=d/r>0.
\]
For $p_1\le1$, the maximum is $1/p_2$, in agreement with $p'_1=\infty$.
The strict inequality~\eqref{eq:2} excludes the simultaneous case
$p_1\le1$, $p_2=\infty$, and the assumptions on the inner indices give
$2<r<\infty$.

Let $\mathcal H_j$ contain the level-$j$ wavelets whose supports meet
$\overline\Omega$. Fix a closed cube $Q\subset\Omega$ with nonempty
interior, and let $\mathcal K_j$ consist of the wavelets of one direction
supported in $Q$. Compact support and the lattice of translations give
\begin{equation}\label{eq:103}
 c2^{jd}\le|\mathcal K_j|\le|\mathcal H_j|\le C2^{jd}
\end{equation}
for all sufficiently large $j$. For a collection $\mathcal L$, let
$b_i(\mathcal L)$ be the restricted coefficient space with
quasi-norm~\eqref{eq:101} at $(s_i,p_i,q_i)$, and let
$J_{\mathcal L}:b_1(\mathcal L)\to b_2(\mathcal L)$ denote the
coefficientwise identity. Multiplication by the two coefficient weights
identifies $J_{\mathcal L}$ with
\begin{equation*}
 D_{\mathcal L}:\left(\bigoplus_j\ell_{p_1}^{|\mathcal L_j|}\right)_{\ell_{q_1}}
 \longrightarrow
 \left(\bigoplus_j\ell_{p_2}^{|\mathcal L_j|}\right)_{\ell_{q_2}},
 \qquad (D_{\mathcal L}x)_j=2^{-jd/r}x_j.
\end{equation*}
Changes in finitely many initial weights give bounded invertible
multipliers, and omitted levels have zero multiplicity.
Proposition~\ref{result:6.1} and Theorem~\ref{result:5.1} therefore apply
to both collections.

Let $P_{\mathcal L},Q_{\mathcal L}$ denote coordinate deletion and zero
insertion, respectively. The wavelets discarded in the following
factorization vanish on $\Omega$, so
\begin{equation}\label{eq:105}
 I_\Omega=R_\Omega W^{-1}Q_{\mathcal H}J_{\mathcal H}
               P_{\mathcal H}WE_\Omega.
\end{equation}
All factors other than $J_{\mathcal H}$ are bounded. The tails of
$J_{\mathcal H}$ tend to zero in norm by geometric decay, so
$J_{\mathcal H}$ is compact, including at infinite fine indices. The factorization therefore
defines a compact Besov embedding and gives the upper estimate.

For the lower estimate, set
\[
 S_1=R_\Omega W^{-1}Q_{\mathcal K},\qquad
 C_2=P_{\mathcal K}WE_\Omega.
\]
Coefficients corresponding to wavelets supported inside $\Omega$ depend only on
the distribution restricted to this domain, including for negative
smoothness. To see this, put
$\sigma_p=d(1/p-1)_+$. For a fixed compact set $K$ and a nonnegative
integer $k>\sigma_p-s$, Littlewood--Paley decomposition gives
\begin{equation*}
 |\langle F,\phi\rangle|
 \le C_K\|F\|_{B_{p,q}^s(\mathbb R^d)}\|\phi\|_{C^k},
 \qquad \phi\in C_c^k(K).
\end{equation*}
Indeed, the $j$th term is bounded by
$C_K2^{j\sigma_p-jk}\|\Delta_jF\|_p\|\phi\|_{C^k}$:
the enlarged frequency projection of $\phi$ has the
$L^{p'}$ or $L^\infty$ bound $C_K2^{-jk}\|\phi\|_{C^k}$. We apply
H\"older's inequality for $p\ge1$ and the band-limited Nikol'skii
inequality $\|\Delta_jF\|_1\lesssim2^{j\sigma_p}\|\Delta_jF\|_p$ for
$p<1$. Summation converges geometrically because
$\sup_j2^{js}\|\Delta_jF\|_p\lesssim\|F\|_{B_{p,q}^s}$.
By smooth approximation in $C^k$ within a compact subset of $\Omega$,
ambient Besov distributions agreeing on $\Omega$ therefore have the
same coefficients against wavelets supported in $Q$; compare the
lower diagram in~\cite[Section~2.4, equations (2.11)--(2.12)]{ref3}.

For $c\in b_1(\mathcal K)$, the distribution
$F=W^{-1}Q_{\mathcal K}c$ also belongs to the ambient target space by
boundedness of $J_{\mathcal K}$. The local coefficient identity applied
to $E_\Omega R_\Omega F$ and $F$ gives
\begin{equation*}
 C_2I_\Omega S_1=J_{\mathcal K}.
\end{equation*}
The ideal property and~\eqref{eq:105} now yield
\begin{equation*}
 a_N(J_{\mathcal K})\lesssim a_N(I_\Omega)
                         \lesssim a_N(J_{\mathcal H}).
\end{equation*}
The factorizations hold on the full coefficient spaces because their
wavelet expansions converge in distributions. By~\eqref{eq:103},
Proposition~\ref{result:6.1} compares both coefficient maps with the
geometric sequence model at the same $N$. Theorem~\ref{result:5.1} then
gives the asserted order with exponent~\eqref{eq:4}. A single interior
block of sufficiently large dimension gives positivity at each fixed
$N$, completing the estimate for all $N\ge1$.
\end{proof}

\begin{remark}\label{result:6.4}
Theorem~\ref{result:1.1} holds for arbitrary real $s_1,s_2$ satisfying the
stated difference condition and concerns approximation numbers of
isotropic, unweighted Besov embeddings. The different ordering of scale
and spatial summation in Triebel--Lizorkin norms requires further
arguments.
\end{remark}

\appendix
\numberwithin{equation}{section}

\section{Basic block estimates}
\label{app:tools}

This appendix proves the auxiliary estimates used in Section~\ref{sec:preliminaries}.
For finite matrices, we use the trace and Hilbert--Schmidt norms $S_1,S_2$ and the inequality
\begin{equation}\label{eq:8}
 |\tr T|\le\|T\|_{S_1}\le\sqrt{\rank T}\,\|T\|_{S_2}.
\end{equation}

\begin{proof}[Proof of Lemma~\ref{result:2.1}]
For $1\le p_1<2<p_2\le p_1'$, the finite-dimensional formula in
\cite[Lemma~3.2]{ref3} gives \eqref{eq:9} with $r=p_2$: for $n\le M/2$, its
factor $\sqrt{1-n/M}$ is bounded away from zero, and the other term in that formula is smaller than the
displayed minimum. If $p_2>p_1'$, apply the finite-dimensional adjoint identity
$a_n(\ell_{p_1}^M\hookrightarrow\ell_{p_2}^M)
=a_n(\ell_{p_2'}^M\hookrightarrow\ell_{p_1'}^M)$ and the preceding case.
For $p_1<1$, the target is Banach and every matrix $F$ satisfies
\begin{equation*}
 \|F:\ell_{p_1}^M\to\ell_{p_2}^M\|
 =\max_k\|Fe_k\|_{p_2}=\|F:\ell_1^M\to\ell_{p_2}^M\|.
\end{equation*}
Indeed, $\|x\|_{p_1}\le1$ implies $\sum_k|x_k|\le1$, and testing on the coordinate vectors
gives equality. Apply this identity to $F=I-R$ and take the infimum over $\rank R<n$.
\end{proof}

\begin{proof}[Proof of Lemma~\ref{result:2.2}]
Choose $j$ with $M_j\ge2N$. The ideal property and Lemma~\ref{result:2.1} give
$a_N(D)\ge t_ja_N(\ell_{p_1}^{M_j}\hookrightarrow\ell_{p_2}^{M_j})>0$.
\end{proof}

\begin{lemma}\label{result:2.3}
Let $T=\bigoplus_jT_j:(\bigoplus_jX_j)_{\ell^{q_1}}\to
(\bigoplus_jY_j)_{\ell^{q_2}}$, where $X_j$ are quasi-Banach and $Y_j$ are Banach.
For an infinite family, assume the standing completeness and uniform
quasi-triangle conditions. If $a_N(T_j)\ge d$ on $L$ selected blocks, then
\begin{equation}\label{eq:12}
 a_N(T)\ge c_{q_2}dL^{\chi_q}.
\end{equation}
\end{lemma}

\begin{proof}
For $\rank A<N$, choose unit vectors $x_j$ in the selected source blocks with
$\|(T_j-A_{jj})x_j\|\ge d/2$, and coefficients $c_j\ge0$ such that
\begin{equation*}
 \|c\|_{q_1}=1,\qquad\|c\|_{q_2}=L^{\chi_q}.
\end{equation*}
Use constant coefficients if $q_2<q_1$, and one nonzero coefficient otherwise.
For independent signs, set $x=\sum_jc_j\varepsilon_j\iota_jx_j$, so $\|x\|\le1$.
Conditioning on all signs except $\varepsilon_j$, we write the $j$-th output as
$\varepsilon_jc_j(T_j-A_{jj})x_j+v_j$. The triangle inequality in $Y_j$ shows
that one sign gives norm at least $c_jd/2$. Thus for $q_2<\infty$,
\[
 \mathbb E\|(T-A)x\|^{q_2}\ge2^{-q_2-1}d^{q_2}\sum_jc_j^{q_2}.
\]
This expectation estimate also applies when $q_2<1$, since the triangle inequality was used only in a target block. If $q_2=\infty$, restrict to
one block. Taking the infimum over $A$ proves \eqref{eq:12}.
\end{proof}

\begin{proof}[Proof of Proposition~\ref{result:2.7}]
For large $N$, choose $J$ maximal with $S_J=\sum_{j\le J}M_j\le N/4$. Then
\begin{equation*}
 2^J\asymp N.
\end{equation*}
Put $L=\lceil(r/2)\log_2(e+N)\rceil$ and $n=\lfloor N/(4L)\rfloor$.
Reproduce the blocks $j\le J$ exactly, approximate each of the next $L$ blocks
with rank less than $n$ by Lemma~\ref{result:2.1}, and set the approximant equal
to zero on the remaining blocks. Since $n\asymp N/L$ and $n\le M_j/2$ on this window, the errors
on these blocks are $\lesssim t_jM_j^{1/r}n^{-1/2}\lesssim n^{-1/2}$, while
\begin{equation*}
 \rank A<S_J+Ln\le N/2<N.
\end{equation*}
The scalar multiplier formula in Section~\ref{sec:preliminaries} gives
\[
 \|D-A\|\lesssim L^{\chi_q}n^{-1/2}+2^{-(J+L)/r}
 \lesssim N^{-1/2}\Lambda_N^{\chi_q+1/2}.
\]
For the lower estimate, take
\begin{equation*}
 \mathcal J_N=\{j:4N\le M_j\le N^{r/2}\}.
\end{equation*}
There are $\asymp\Lambda_N$ such levels. Lemma~\ref{result:2.1} gives
$a_N(t_jI_{M_j})\gtrsim N^{-1/2}$ on each of these levels. The dimension factors
cancel in the second term of the minimum, while $M_j\le N^{r/2}$ gives the same lower bound for the first.
Lemma~\ref{result:2.3} proves the lower bound. Positivity and boundedness of $D$
extend the estimates to the finitely many remaining indices.
\end{proof}

\begin{proof}[Proof of Lemma~\ref{result:trace-window}]
Choose a fixed stride $h$ and $Q>1$ so that, after finitely many levels,
\begin{equation*}
 3^rM_j\le M_{j+h}\le QM_j.
\end{equation*}
Choose $c>0$ with
\begin{equation*}
 c\log Q<(r-2)/4,
\end{equation*}
set $L=\lfloor c\Lambda_N\rfloor$, and take $j_1$ to be the first level on this
subsequence with $M_{j_1}\ge C_*N/L^2$. This gives \eqref{eq:37}, and
\begin{equation*}
 m_L\le Q^{L-1}m_1\lesssim N^{r/2-\eta}L^{-2}
\end{equation*}
for some $\eta>0$. The polynomial gain absorbs every fixed power of $L$;
therefore $m_L\lesssim N^{r/2}L^{-r\Gamma}$, proving \eqref{eq:39}.

For the trace assertion, fix $\rank A<N$ and $E=D_L-A$, $\varepsilon=\|E\|$.
If $\varepsilon\ge t_L/2$, use \eqref{eq:39}. Otherwise each source coordinate
vector has norm one and each target coordinate functional has norm one, so
$|E_{ii}|\le\varepsilon<t_s/2$ in block $s$. For the set in
\eqref{eq:trace-input}, put $R_S=B_SD_{SS}C_S$ and $T_S=B_SA_{SS}C_S$. Then
\begin{equation*}
 R_S=\bigoplus_{s\in S}m_s^{-1}I_{m_s},\quad
 \|R_S\|_{S_2}^2=\sum_{s\in S}m_s^{-1},\quad
 \operatorname{Re}\tr T_S\ge|S|/2,\quad\rank T_S<N.
\end{equation*}
By \eqref{eq:8}, \eqref{eq:37}, and \eqref{eq:trace-input},
\[
 |S|/2\le\sqrt N\|T_S\|_{S_2}
 \lesssim L/\sqrt{C_*}+K\sqrt N\,\varepsilon L^{1-\Gamma}.
\]
Choose $C_*$ large enough to absorb the first term. Since $|S|\ge7L/8$,
\begin{equation*}
 \varepsilon\gtrsim N^{-1/2}L^\Gamma.
\end{equation*}
Taking the infimum over the approximants proves the assertion. Compressing to the selected levels also gives
\begin{equation*}
 a_N(D)\ge a_N(D_L).\qedhere
\end{equation*}
\end{proof}

\subsection{Column estimates}

\begin{lemma}\label{result:3.1}
Let $0<s\le\vartheta\le1$, let $X_j$ satisfy the standing assumptions, and let
$Y$ be complete with $\|y+z\|_Y^\vartheta\le\|y\|_Y^\vartheta+\|z\|_Y^\vartheta$.
If bounded operators $T_s:(\bigoplus_jX_j)_{\ell^s}\to Y$ and
$T_\vartheta:(\bigoplus_jX_j)_{\ell^\vartheta}\to Y$ agree on finite sequences,
then
\begin{equation*}
 a_N(T_s)=a_N(T_\vartheta).
\end{equation*}
\end{lemma}

\begin{proof}
The contractive inclusion gives $a_N(T_s)\le a_N(T_\vartheta)$. For the reverse inequality, take
$\rank A<N$ on the $\ell^s$-sum and put $E=T_s-A$ and $\varepsilon=\|E\|$.
The series $\widetilde Ex=\sum_jE\iota_jx_j$ converges in $Y$ on the
$\ell^\vartheta$-sum and satisfies
$\|\widetilde Ex\|_Y^\vartheta\le\varepsilon^\vartheta\sum_j\|x_j\|^\vartheta$.
On finite sequences, $T_\vartheta-\widetilde E=A$. By density and the closedness of the
finite-dimensional range of $A$, the operator $T_\vartheta-\widetilde E$ has rank
less than $N$. Thus $a_N(T_\vartheta)\le\varepsilon$.
\end{proof}

\begin{lemma}\label{result:3.2}
Let $0<q\le1$, let $\Omega_s$ be nonempty finite sets, and let
$v_{\ell s,\omega}\in H_\ell$, where $H_\ell$ are real or complex Hilbert spaces.
Put
\begin{equation}\label{eq:46}
 \mathfrak M_q=\max_{\omega_s\in\Omega_s,\ \varepsilon_s\in\{-1,1\}}
 \left(\sum_{\ell=1}^L\left\|\sum_s\varepsilon_sv_{\ell s,\omega_s}\right\|^q\right)^{1/q}.
\end{equation}
There is a set $S$ of indices, used for both rows and columns, with $|S|\ge7L/8$ such that
\begin{equation}\label{eq:47}
 \sum_{\ell,s\in S}\frac1{|\Omega_s|}\sum_{\omega\in\Omega_s}
 \|v_{\ell s,\omega}\|^2\le C_qL^{1-2/q}\mathfrak M_q^2.
\end{equation}
\end{lemma}

\begin{proof}
Set $g_{\ell s}=\max_\omega\|v_{\ell s,\omega}\|$. Conditioning on all signs except the one attached to a matched
source index and applying the triangle inequality shows that every bipartite matching
$\mathcal Q$ satisfies
\begin{equation}\label{eq:48}
 \sum_{(\ell,s)\in\mathcal Q}g_{\ell s}^q\le2\mathfrak M_q^q.
\end{equation}
If $\mathfrak M_q=0$, the conclusion holds. Otherwise, set
\begin{equation*}
 \tau=(128\mathfrak M_q^q/L)^{1/q}.
\end{equation*}
The graph $g_{\ell s}>\tau$ has matching number less than $L/64$.
Apply K\"onig's matching--cover theorem. From both the row and column index sets, delete all labels in the
resulting vertex cover. The remaining set satisfies
\begin{equation*}
 |S_0|>63L/64,\qquad g_{\ell s}\le\tau\quad(\ell,s\in S_0).
\end{equation*}
The functional \eqref{eq:46} restricted to $S_0\times S_0$ is at most
$2^{1/q-1}\mathfrak M_q$. Indeed, the signed sum $V$ over the omitted sources is symmetric, and
$\|u+V\|^q+\|u-V\|^q\ge2^q\|u\|^q$; averaging and summing over the retained targets gives this bound.

Choose the $\omega_s$ independently and uniformly. Put
$X_\ell=\sum_{s\in S_0}\|v_{\ell s,\omega_s}\|^2$, $\mu_\ell=\mathbb EX_\ell$,
and $a=q/2$. The Hilbert-space lower Khintchine inequality gives
\begin{equation*}
 \sum_{\ell\in S_0}\mathbb EX_\ell^a\le C_q\mathfrak M_q^q.
\end{equation*}
Independence and the bound $\|v\|\le\tau$ give
\begin{equation*}
 \mathbb EX_\ell^2\le\mu_\ell^2+\tau^2\mu_\ell.
\end{equation*}
Interpolating the moments of orders $a,1,2$ yields
\begin{equation*}
 \mathbb EX_\ell^a\ge h_\ell:=
 \frac{\mu_\ell}{(\mu_\ell+\tau^2)^{1-a}},\qquad
 \sum_{\ell\in S_0}h_\ell\le C_q\mathfrak M_q^q.
\end{equation*}
For $L\ge16$, delete the labels corresponding to the $\lfloor L/16\rfloor$ largest nonzero $h_\ell$ (all such labels if
there are fewer), leaving $S$ with $|S|\ge7L/8$ and
\begin{equation*}
 \max_{\ell\in S}h_\ell\le C_q\mathfrak M_q^q/L\lesssim_q\tau^q.
\end{equation*}
Since $y/(1+y)^{1-q/2}$ increases from zero to infinity, this implies
$\mu_\ell\lesssim_q\tau^2$ on $S$. Therefore the sum in \eqref{eq:47} is
at most $\sum_{\ell\in S}\mu_\ell\lesssim_qL\tau^2$, as required.
For $L<16$, take all labels and apply \eqref{eq:48} to single edges.
\end{proof}

\begin{proof}[Proof of Proposition~\ref{result:3.3}]
We first establish the weighted error estimate on a finite window. Let
$0<p\le r'$, $0<q\le1$, $q\le u\le\infty$, and
\begin{equation*}
 0<c_0\le d_s:=t_sm_s^{1/r}\le C_0.
\end{equation*}
For $E:(\bigoplus_s\ell_p^{m_s})_{\ell^u}\to
(\bigoplus_\ell\ell_r^{m_\ell})_{\ell^q}$ put
\begin{equation}\label{eq:56}
 B=\bigoplus_\ell m_\ell^{1/r-1/2}I_{m_\ell},\qquad
 C=\bigoplus_s d_s^{-1}m_s^{-1/2}I_{m_s}.
\end{equation}
Take $v_{\ell s,k}=m_\ell^{1/r-1/2}d_s^{-1}E_{\ell s}e_k\in\ell_2^{m_\ell}$,
so
\begin{equation*}
 \|B_SE_{SS}C_S\|_{S_2}^2
 =\sum_{\ell,s\in S}\frac1{m_s}\sum_{k=1}^{m_s}\|v_{\ell s,k}\|_2^2.
\end{equation*}
The vector with blocks $x_s=d_s^{-1}\varepsilon_se_{k_s}$ has norm at most
$c_0^{-1}L^{1/u}$. Since $\|m_\ell^{1/r-1/2}y\|_2\le\|y\|_r$,
\begin{equation*}
 \mathfrak M_q\le c_0^{-1}L^{1/u}\|E\|.
\end{equation*}
Lemma~\ref{result:3.2} gives a common set satisfying $|S|\ge7L/8$ and
\begin{equation}\label{eq:57}
 \|B_SE_{SS}C_S\|_{S_2}\le C_{r,q,c_0}\|E\|L^{1/2-1/q+1/u}.
\end{equation}
This estimate applies to error operators of arbitrary rank, including operators with entries between different blocks.

If $q_1\ge q_2$, apply Lemma~\ref{result:trace-window} with
$\Gamma=1/2+1/q_2-1/q_1$, $u=q_1$, $q=q_2$, and the weights \eqref{eq:56}.
They satisfy $b_sc_st_s=1/m_s$, and \eqref{eq:57} is
\eqref{eq:trace-input}. This proves the lower bound, including the finite-window
statement \eqref{eq:column-window} when $u=\infty$. The upper bound follows from
Proposition~\ref{result:2.7}. If $q_1<q_2\le1$, apply Lemma~\ref{result:3.1}
with $s=q_1$, $\vartheta=q_2$ to obtain
$a_N(D_{q_1,q_2})=a_N(D_{q_2,q_2})$; the target has the required
$q_2$-power triangle inequality. The equal-index case gives exponent $1/2$.
Positivity and boundedness extend the estimates to the finitely many small $N$.
\end{proof}

\subsection{Equal conjugate inner indices}

\begin{proof}[Proof of Proposition~\ref{result:2.8}]
The upper bound follows from Proposition~\ref{result:2.7}. For $q_1\le q_2$, take the
window of Lemma~\ref{result:trace-window} with $\Gamma=1/2$ and set
$\lambda_s=(m_st_s)^{-1/2}$, $\Lambda=\bigoplus_s\lambda_sI_{m_s}$.
For any matrix $E$ with block norms
$e_{\ell s}=\|E_{\ell s}:\ell_{r'}^{m_s}\to\ell_r^{m_\ell}\|$, coordinate
injections and projections give $e_{\ell s}\le\|E\|$. Applying the Euclidean column
bound to $E_{\ell s}$ and its adjoint gives
\begin{align*}
 \|E_{\ell s}\|_{S_2}^2&\le m_sm_\ell^{1-2/r}e_{\ell s}^2,
 \\
 \|E_{\ell s}\|_{S_2}^2&\le m_\ell m_s^{1-2/r}e_{\ell s}^2.
\end{align*}
Thus $t_s\asymp m_s^{-1/r}$ and geometric separation imply
\begin{equation*}
 \|\lambda_\ell E_{\ell s}\lambda_s\|_{S_2}^2
 \lesssim\left(\frac{\min(m_\ell,m_s)}{\max(m_\ell,m_s)}\right)^{1/r}e_{\ell s}^2
 \lesssim3^{-|\ell-s|}\|E\|^2.
\end{equation*}
Summing over the matrix blocks gives
\begin{equation*}
 \|\Lambda E\Lambda\|_{S_2}\lesssim L^{1/2}\|E\|.
\end{equation*}
Apply the trace lemma with $S$ the entire window and $B=C=\Lambda$ to obtain
$a_N(D)\gtrsim N^{-1/2}\Lambda_N^{1/2}$.

For $q_1\ge q_2$, put $q_* =\min\{1/2,q_2\}$, taking $q_*=1/2$ at
$q_2=\infty$. Proposition~\ref{result:3.3} supplies a window of
$L\asymp\Lambda_N$ levels on which
$a_N(D_{\infty,q_*}^{\mathcal W_N})\gtrsim N^{-1/2}L^{1/2+1/q_*}$.
On that same window the inclusions $J:\ell^\infty\to\ell^{q_1}$ and
$K:\ell^{q_2}\to\ell^{q_*}$, with the inner indices fixed, satisfy
\[
 D_{\infty,q_*}^{\mathcal W_N}=K D_{q_1,q_2}^{\mathcal W_N}J,\qquad
 \|J\|=L^{1/q_1},\quad\|K\|=L^{1/q_*-1/q_2}.
\]
The ideal property therefore gives
$a_N(D)\gtrsim N^{-1/2}L^{1/2+1/q_2-1/q_1}$, as required.
The approximating matrices in this finite-window argument may have entries between different levels. Positivity
extends the lower bound to the finitely many small indices.
\end{proof}

\section{Lower bounds from Malykhin's finite-window estimate}
\label{app:malykhin}

We derive the lower estimates for the two strict orderings of the inner indices from
Malykhin's argument~\cite[proof of Theorem~1, especially
equation~(13) in the author version]{ref13}. We obtain the approximation-number
bounds by comparing norms, taking convex combinations in the source and using finite-dimensional
adjoints. The argument covers quasi-Banach indices and both real and complex scalars.

\begin{proposition}\label{app:malykhin:consequences}
Let $D_{q_1,q_2}$ be the critical map~\eqref{eq:7}. For
$1\le q_1,q_2\le\infty$,
\[
 a_N(D_{q_1,q_2})\gtrsim N^{-1/2}\Lambda_N^{B_t(q_1,q_2)}
 \quad\text{if }p_2=r\le p_1',
\]
and
\[
 a_N(D_{q_1,q_2})\gtrsim N^{-1/2}\Lambda_N^{B_s(q_1,q_2)}
 \quad\text{if }p_1'=r\le p_2,
\]
where
\begin{equation}
 B_t(u,v)=\max\{1/2,1-1/u\}-(1/2-1/v)_+,
 \label{app:malykhin:target-exponent}
\end{equation}
\begin{equation}
 B_s(u,v)=\max\{1/2,1/v\}-(1/u-1/2)_+.
 \label{app:malykhin:source-exponent}
\end{equation}
For $0<q_1<1\le q_2$, the same estimates hold with $q_1$ replaced
by $1$. For every admissible inner pair and
$0<q_1\le q_2\le1$,
\[
 a_N(D_{q_1,q_2})\gtrsim N^{-1/2}\Lambda_N^{1/2}.
\]
These lower bounds first hold on finite windows and then pass
to the full coefficient spaces.
\end{proposition}

\begin{proof}
We first consider the inner indices $1,p$, where $2<p<\infty$, and block
dimensions $2^j$. Fix $2<\widetilde p<p$ and, for large $N$, set
\begin{equation*}
 \mathcal W_N=\{j:4N<2^j<N^{\widetilde p/2}\},\qquad
 L=|\mathcal W_N|\asymp\Lambda_N.
\end{equation*}
On the real coordinates of this window, define
\begin{align*}
 K_\theta&=\{x:\|(\|x_j\|_1)_j\|_\theta\le1\},
       &&1\le\theta\le\infty,\\
 \|x\|_{Y_s}&=\|(2^{-j/p}\|x_j\|_p)_j\|_s,
       &&1\le s\le\infty,\\
 \|x\|_Z&=\max\left\{
   \Big(\sum_j2^{-j}\|x_j\|_p^p\Big)^{1/p},
   \Big(\sum_j2^{-j}\|x_j\|_2^2\Big)^{1/2}\right\}.
\end{align*}
Write $d_n(K,V)=\inf_{\dim F\le n}\sup_{x\in K}\inf_{y\in F}
\|x-y\|_V$. Malykhin's finite-window proof gives
\begin{equation}
 d_N(K_\theta,Z)\gtrsim
 N^{-1/2}L^{\beta(\theta)},\qquad
 \beta(\theta)=\max\{1/2,1-1/\theta\}.
 \label{app:malykhin:finite-lower}
\end{equation}
Indeed, each level has total coordinate weight $2^j2^{-j}=1$.
For the product of level octahedra, the block-weight ratios in
\cite[Theorem~2 and equation~(1)]{ref13} are $1/L$; the other
hypotheses hold for large $N$, since $L\log(2L)=o(N)$.
The two terms in that estimate are $cL/\sqrt N$ and
$2^{-\max\mathcal W_N/p}/2$. The second term is larger than the first for large $N$
because $\widetilde p<p$, proving~\eqref{app:malykhin:finite-lower}
for $\theta=\infty$. Regarding all coordinates as one octahedron
gives $c\sqrt{L/N}$ for $\theta=1$. The inclusions
$K_1\subset K_\theta$ and
$L^{-1/\theta}K_\infty\subset K_\theta$ give the estimate for every remaining
$\theta$.

Put $a_j=2^{-j/p}\|x_j\|_p$. Since $p>2$,
$\|a\|_p\le\|a\|_2$ and
$2^{-j/2}\|x_j\|_2\le a_j$. Consequently
\begin{equation*}
 \|x\|_Z\le\|a\|_2
 \le L^{(1/2-1/s)_+}\|x\|_{Y_s}.
\end{equation*}
Apply this comparison to each error $x-y$. The range of a rank-$<N$
linear approximant is admissible for $d_{N-1}$, so
$a_N\ge d_{N-1}\ge d_N$. The target isometry
$x\mapsto(2^{-j/p}x_j)_j$ therefore yields
\begin{equation}
 a_N(D_{\theta,s}^{\mathcal W_N})\gtrsim
 N^{-1/2}\Lambda_N^{\beta(\theta)-(1/2-1/s)_+},
 \qquad1\le\theta,s\le\infty,
 \label{app:malykhin:basic-lower}
\end{equation}
for inner indices $1,p$.

For complex scalars, restrict a complex-linear rank-$<N$
approximant to real inputs and take real parts of its output.
Its real rank is $<2N$ and its error on real inputs does not
increase. Thus $a_N^{\mathbb C}(D)\ge a_{2N}^{\mathbb R}(D)$;
applying the real bound on $\mathcal W_{2N}$ gives the same order.
The subsequent finite-dimensional comparisons work over either
field.

Take $p_2=r\le p_1'$. If $p_1\ge1$, the contractive source
inclusion $\ell_1\hookrightarrow\ell_{p_1}$ transfers
\eqref{app:malykhin:basic-lower} to these inner indices. If
$p_1<1$, Lemma~\ref{result:4.4}, applied on the finite window, gives the same approximation numbers as
for $p_1=1$, since $q_2\ge1$ makes the target Banach. The resulting
exponent is $B_t$ in~\eqref{app:malykhin:target-exponent}.

If $p_1'=r\le p_2$ and $q_1,q_2\ge1$, take finite-dimensional
Banach adjoints. Their inner indices are $p_2',r$ and their fine
indices are $q_2',q_1'$, so the preceding argument applies to the
adjoints. This includes $p_2=\infty$, with $p_2'=1$.
Approximation numbers are unchanged by finite-dimensional
adjoints, including the $1,\infty$ endpoints. Substituting
$(q_2',q_1')$ into $B_t$ gives $B_s$ in
\eqref{app:malykhin:source-exponent}.

For $q_1<1\le q_2$, apply Lemma~\ref{result:4.4} to the source outer sum.
The target is Banach, so this reduces the problem to $q_1=1$. At $(q_1,q_2)=(1,1)$, both formulas
give a finite-window lower bound with exponent $1/2$ for every
admissible inner pair. For $0<q_1\le q_2\le1$, include the target
outer space contractively in $\ell_1$ and then apply the same lemma to the source outer
sum with this Banach target. The same lower bound follows.

Finally, the coordinate factorizations in
Proposition~\ref{result:6.1}, with $d=1$, transfer the bounds
from $2^j$ to $M_j\asymp2^j$ at the same index $N$.
Contractive compression then passes the finite-window lower bounds to the
full spaces. In particular, they apply to the
full spaces of bounded block sequences at infinite fine indices.
\end{proof}

\section*{Declarations}

\noindent\textbf{Competing interests.} The author declares no competing interests.

\medskip
\noindent\textbf{Data availability.} No datasets were generated or analysed during this study.

\medskip
\noindent\textbf{AI Statement.} The author used artificial-intelligence tools for language editing,
\LaTeX{} formatting, and limited assistance with local mathematical reasoning. The
proof strategy and mathematical development are the author's own; all mathematical
content was independently written and checked by the author, who takes full
responsibility for it.

\end{document}